\documentclass[12pt]{amsart}
\usepackage[a4paper,hmargin=3.5cm,vmargin=4cm]{geometry}
\usepackage{graphicx} 
\usepackage{verbatim,enumerate}
\usepackage[dvipsnames]{xcolor}
\usepackage[colorlinks]{hyperref}
\usepackage{ulem}
\graphicspath{{./Figures/}}
\usepackage {svg}
\usepackage{wrapfig}
\usepackage{amsrefs}
\usepackage{amssymb, amsmath}
\usepackage{bigints}
\usepackage{here}

\usepackage{pb-diagram}

\title[Krust-type theorem and deformations of minimal surfaces]{%
Extension of Krust theorem and deformations of minimal surfaces
}

\author[S.~Akamine and H.~Fujino]{
Shintaro Akamine and Hiroki Fujino
}   

\address[Shintaro Akamine]{%
Department of Liberal Arts, College of Bioresource Sciences,
Nihon University, 
1866 Kameino, Fujisawa, Kanagawa, 252-0880, Japan}
\email{akamine.shintaro@nihon-u.ac.jp}

\address[Hiroki Fujino]{%
  Institute for Advanced Research, Graduate School of Mathematics, 
Nagoya University, Chikusa-ku, Nagoya 464-8602, Japan
}
\email{m12040w@math.nagoya-u.ac.jp}



\thanks{
The first author was partially supported by 
JSPS KAKENHI Grant Number 19K14527,
and the second author by JSPS KAKENHI Grant Number 20K14306.
}

\usepackage{amsthm}
\theoremstyle{plain}

 \newtheorem{theorem}{Theorem}[section]
 \newtheorem{proposition}[theorem]{Proposition}

 \newtheorem{lemma}[theorem]{Lemma}
 \newtheorem{corollary}[theorem]{Corollary}

\theoremstyle{definition}
 
\theoremstyle{remark}
 \newtheorem{remark}[theorem]{Remark}
 \newtheorem*{remark*}{Remark}
\newtheorem{example}[theorem]{Example}
 \newtheorem*{acknowledgement}{Acknowledgement}

\numberwithin{equation}{section}

\renewcommand{\Re}{\operatorname{Re}}

\renewcommand{\phi}{\varphi}

\newcommand{\Rc}{\mathbb{R}^3(c)}
\newcommand{\RZ}{\mathbb{R}/2\pi \mathbb{Z}}

\newcommand{\transpose}[1]{%
\,  {\vphantom{#1}}^{t}\!{#1}  
}

\definecolor{Blue}{rgb}{0,0,1}  
\definecolor{Red}{rgb}{1,0,0}  

\begin{document}
\maketitle

\begin{abstract}
In the minimal surface theory, the Krust theorem asserts that if a minimal surface in the Euclidean 3-space $\mathbb{E}^3$ is the graph of a function over a convex domain, then each surface of its associated family is also a graph. The same is true for maximal surfaces in the Minkowski 3-space $\mathbb{L}^3$.

In this article, we introduce a new deformation family that continuously connects minimal surfaces in $\mathbb{E}^3$ and maximal surfaces in $\mathbb{L}^3$, and prove a Krust-type theorem for this deformation family. This result induces Krust-type theorems for various important deformation families containing the associated family and the L\'opez-Ros deformation.

Furthermore, minimal surfaces in the isotropic 3-space $\mathbb{I}^3$ appear in the middle of the above deformation family. We also prove another type of Krust's theorem for this family, which implies that the graphness of such minimal surfaces in $\mathbb{I}^3$ strongly affects the graphness of deformed surfaces.

The results are proved based on the recent progress of planar harmonic mapping theory.
\end{abstract}


\section{Introduction} \label{sec:1} 
Minimal surfaces in the Euclidean $3$-space $\mathbb{E}^3$ are interesting objects in the classical differential geometry, and many researchers invented various kinds of deformations of minimal surfaces depending on their respective purposes. In their researches, to observe embeddedness of deformed minimal surfaces often plays an important role, but is generally non-trivial. For example, L\'opez-Ros \cite{LR} gave a characterization of the plane and catenoid as embedded complete minimal surfaces of finite total curvature and genus zero. In their method, it was essential to see embeddedness of some deformation which is now called the {\it L\'opez-Ros deformation} (or the {\it Goursat transformation}).

For another example, the Krust theorem stated below also played an essential role in the conjugate construction of embedded saddle tower by Karcher \cite{Kar}. The Krust theorem deals with graphness of minimal surfaces for a deformation family called the {\it associated family} (or the {\it Bonnet transformation}). Here, note that the embeddedness and graphness are closely related to each other since any surface which can be written as a graph is embedded.

\begin{theorem}[Krust, \cite{Kar} or {\cite[p.122]{DHS}}]\label{thm:Krust}
If a minimal surface is a graph over a convex domain, then each surface of its associated family is also a graph.
\end{theorem}

On the other hand, as has often been pointed out, the simultaneous consideration of  minimal surfaces in $\mathbb{E}^3$ and maximal surfaces in the Minkowski $3$-space $\mathbb{L}^3$ leads to interesting results. Calabi \cite{C} proved a Bernstein-type theorem for maximal surfaces in $\mathbb{L}^3$ by using a one-to-one correspondence between minimal and maximal surfaces, which is known as the classical duality or the {\it Calabi correspondence}. In that context, the Lorentzian version of Theorem \ref{thm:Krust} was also proved in \cite{RLopez}.

In this article, we first introduce a more general form of deformation family as below:
\begin{equation} \label{eq:extended_family_intro}
	X_{\theta, \lambda, c}=\Re \bigints^w \left( \begin{array}{c} 1-c\lambda^2 G^2 \\ -i(1+c\lambda^2 G^2) \\ 2\lambda G \end{array} \right) \frac{e^{i\theta}}{\lambda}Fdw.
\end{equation}
Here, $(F,G)$ is a so-called {\it Weierstrass data} of $X_{\theta, \lambda, c}$ defined in Section \ref{subsec:w_rep_in_Rc}, and $(\theta,\lambda,c)\in \mathcal{P}:=\mathbb{R}/2\pi\mathbb{Z} \times (0,+\infty)\times \mathbb{R}$. If we set $\theta = 0$ and $\lambda = 1$, then it turns out that the formula \eqref{eq:extended_family_intro} unifies each of the Weierstrass representation formulas for (i) minimal surfaces in the Euclidean 3-space $\mathbb{E}^3$ ($c=1$), (ii) maximal surfaces in the Minkowski 3-space $\mathbb{L}^3$ ($c= -1$), and (iii) minimal surfaces in the isotropic 3-space $\mathbb{I}^3$ ($c= 0$). For such surfaces in $\mathbb{I}^3$, see \cite{SY,Pember,MaEtal,Sato,Si,Strubecker}. More generally, we can see that each $X_{\theta, \lambda, c}$ is a (possibly singular) zero mean curvature surface in $\Rc := (\mathbb{R}^3, dx^2+dy^2+cdz^2)$ which is isometric to $\mathbb{E}^3$ if $c>0$, $\mathbb{L}^3$ if $c<0$, and is nothing but $\mathbb{I}^3$ if $c=0$ (see Remark \ref{rmk:isometry}). In this sense, the parameter $c$ plays a role that connects $\mathbb{E}^3$ and $\mathbb{L}^3$ continuously, and this kind of techniques can also be seen in other researches (for example, see \cite{Danciger,Pember,UY92}).  
We emphasise that this parameter $c$ leads to remarkable results, in particular, in Section \ref{sec:Krust_2}.

Furthermore, the deformation family $X_{\theta, \lambda, c}$ includes many of historically significant deformations, see Fig.  \ref{fig:def_family_intro}. In fact, the parameter $\theta$ and $\lambda$ stand for the deformations of the associated family and the L\'opez-Ros deformation, respectively. In addition, we can see that this deformation family also contains the above classical duality correspondence.

\begin{figure}[htbp]
       \begin{center}
           \includegraphics[scale=0.37]{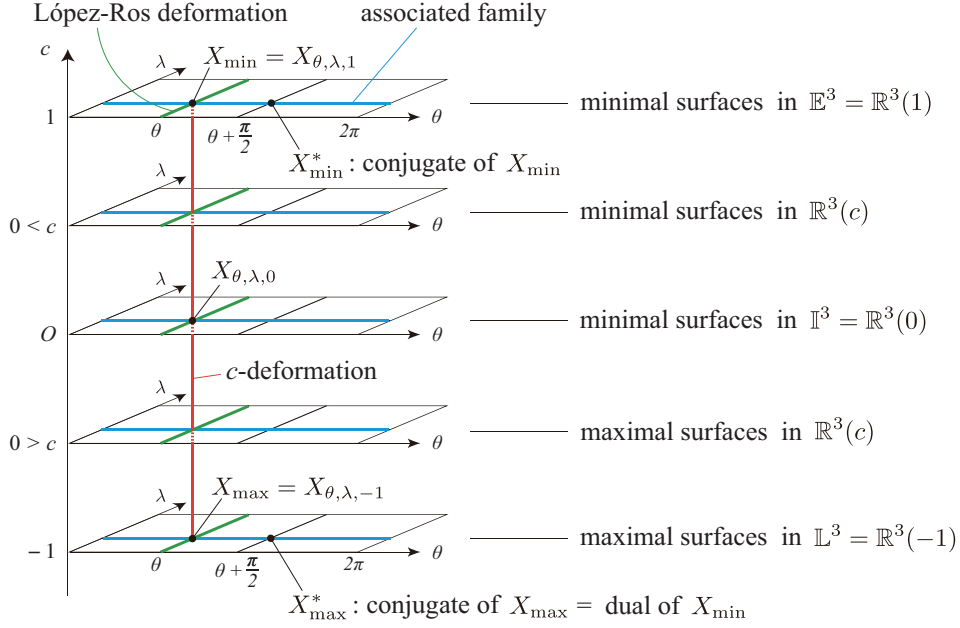} 
       \end{center}
       \caption{$X_{\theta,\lambda,c}$ contains various kinds of deformations.}   \label{fig:def_family_intro}
\end{figure}

In Section \ref{sec:Krust_1}, we prove a Krust-type theorem for the deformation family $X_{\theta,\lambda,c}$. Let us denote the image of $X_{\theta,\lambda,c}$ by $S_{\theta,\lambda,c}$. 

\begin{theorem}\label{thm:main_krust1_intro}
For the Weierstrass data $(F,G)$ of $X_{\theta,\lambda,c}$, suppose that $G$ is not constant and $|G|<1$. If there exists $(\theta_0,\lambda_0,c_0)\in \mathcal{P}$ such that $|c_0 \lambda_0^2|\leq 1/\| G \|_{\infty}^2$ and $S_{\theta_0,\lambda_0,c_0}$ is a graph over a convex domain, then $S_{\theta,\lambda,c}$ is a graph over a close-to-convex domain whenever $(\theta,\lambda,c)\in\mathcal{P}$ satisfies $|c\lambda^2|\leq |c_0 \lambda_0^2|$.
\end{theorem}

\noindent
As a corollary, by taking $(\theta_0,\lambda_0,c_0)=(0,1,1)$, Theorem \ref{thm:main_krust1_intro} simultaneously induces classical Krust's theorem (see Theorem \ref{thm:Krust}) and the Krust-type theorem for the L\'opez-Ros deformation obtained by Dorff in \cite[Corollary 3.5]{Dorff}. 

Furthermore, we give an another type of Krust's theorem for $X_{\theta,\lambda,c}$ in Section \ref{sec:Krust_2}, where we can see that the graphness of the minimal surfaces $X_{\theta,\lambda,0}$ in $\mathbb{I}^3$ strongly affects the graphness of $X_{\theta,\lambda,c}$.

\begin{theorem}\label{thm:main_krust2_intro}
Suppose that $G$ is not constant and $|G|<1$. If the minimal surface $S_{\theta_0,\lambda_0,0}$ in $\mathbb{I}^3$ is a graph over a convex domain for some $(\theta_0,\lambda_0)$, then $S_{\theta,\lambda,c}$ is a graph over a close-to-convex domain for any $(\theta,\lambda,c)\in\mathcal{P}$ with $|c\lambda^2|\leq 1/\| G \|_{\infty}^2$. In particular, the minimal and the maximal surfaces $S_{\theta, 1, \pm 1}$ are graphs.
\end{theorem}

\noindent
Owing to this result, we can also prove the following Krust-type theorem which can discuss the graphness of the surfaces $S_{\theta,\lambda,c}$ even when the original minimal surface does not satisfy the convexity assumption in Theorem \ref{thm:Krust}.

\begin{corollary}\label{cor:rephrasing_Krust2_intro}
Assume that the Weierstrass data $(F,G)$ satisfies that 
\vspace*{0.5em}
\begin{enumerate}\setlength{\parskip}{0.5em}
	\item[$(\mathrm{i})$] $G$ is not constant and $|G|<1$,
	\item[$(\mathrm{ii})$] $\displaystyle h:=\int^w F dw$ is univalent and its image is convex.
\end{enumerate}
\vspace*{0.5em}
\noindent
Then, $S_{\theta,\lambda,c}$ is a graph over a close-to-convex domain for any $(\theta,\lambda,c)\in\mathcal{P}$ with $|c\lambda^2|\leq 1/\| G \|_{\infty}^2$.
\end{corollary} 

\noindent 
We also see in Section \ref{subsec:non_graph} that the estimation $|c\lambda^2|\leq 1/\| G \|_{\infty}^2$ is optimal.

The above theorems and the corollary are obtained in completely different ways from the proof of classical Krust's theorem. To prove them, we fully utilize recent developments of the planar harmonic mapping theory.\\

At the end of the introduction, we give the organization of this paper:

In Section \ref{sec:2}, we give a Weierstrass-type representation formula for zero mean curvature surfaces in $\Rc$. In addition to this, we generalize the concepts of the associated family, the L\'opez-Ros deformation, and the classical duality correspondence to $\Rc$.

In Section \ref{sec:def_family}, by combining all of the deformations defined in Section \ref{sec:2}, we introduce a more general deformation family. After this, we formulate the notations and notion of graphness. Further, we explain the connection between the surface theory and the planar harmonic mapping theory.

As already mentioned above, we prove Theorem \ref{thm:main_krust1_intro} and Theorem \ref{thm:main_krust2_intro} in Section \ref{sec:Krust_1} and Section \ref{sec:Krust_2}, respectively. Moreover, many of important corollaries are explained in each section. In particular, the sharpness of the estimation of the main theorems are discussed in Section \ref{subsec:non_graph}.

Finally, we give examples in Section \ref{sec:example} to see how to apply the main theorems.

\section{Preliminaries} \label{sec:2} 
In this section, we give a notion of deformations of minimal surfaces passing through different ambient spaces.

\subsection{Weierstrass-type representation in different ambient spaces} \label{subsec:w_rep_in_Rc}

 Let us denote the 3-dimensional vector space $\mathbb{R}^3$ with the metric $\langle\ ,\ \rangle_c=dx^2+dy^2+cdt^2$ by $\mathbb{R}^3(c)$, where $(x,y,t)$ are the canonical coordinates of $\mathbb{R}^3$ and $c\in \mathbb{R}$ is a parameter.
 
 Let $\Sigma$ be a Riemann surface and let $X={}^t(X_1,X_2,X_3)\colon \Sigma \to \mathbb{R}^3(c)$ be a non-constant harmonic mapping. Suppose that at any point $p\in \Sigma$ there exists a complex coordinate neighbourhood $(D, w)$ such that the derivatives $\phi_j=\partial{X_j}/\partial{w}, j=1,2,3$ satisfy
 \begin{equation}\label{eq:conformal_regular}
\phi_1^2+\phi_2^2+c\phi_3^2=0,\quad  |\phi_1|^2+|\phi_2|^2+c|\phi_3|^2\not \equiv 0.
 \end{equation}
If we put $w=u+iv$, then the relation \eqref{eq:conformal_regular} represents 
\[
\Big{\langle} \frac{\partial{X}}{\partial{u}},\frac{\partial{X}}{\partial{u}} \Big{\rangle}_c=\Big{\langle} \frac{\partial{X}}{\partial{v}},\frac{\partial{X}}{\partial{v}} \Big{\rangle}_c(\geq 0)\quad \text{ and }\quad \Big{\langle} \frac{\partial{X}}{\partial{u}},\frac{\partial{X}}{\partial{v}} \Big{\rangle}_c =0.
\]
Hence we can see that the induced metric of $X\colon\Sigma \to \mathbb{R}^3(c)$ is positive definite except at some singular points, which corresponds to the spacelike condition of $X$ in $\mathbb{R}^3(c)$ with $c\leq 0$. The harmonicity of $X$ implies that the mean curvature of $X$ vanishes identically. 
Then $X\colon\Sigma \to \mathbb{R}^3(c)$ is said to be a {\it generalized zero mean curvature surface} in $\mathbb{R}^3(c)$, which is a surface in $\mathbb{R}^3(c)$ whose mean curvature vanishes identically possibly with singular points. 
As the special cases, this notion is known in \cite{O} for minimal surfaces with branch points in the Euclidean 3-space $\mathbb{R}^3(1)=\mathbb{E}^3$ and in \cite{ER} for maximal surfaces with singularities in the Minkowski 3-space $\mathbb{R}^3(-1)=\mathbb{L}^3$. It also contains a notion of zero mean curvature surfaces possibly with singular points in the isotropic 3-space $\mathbb{R}^3(0)=\mathbb{I}^3$ (see also \cite{Sato,Si}).
   From now on, unless there is confusion, we will omit the word ``generalized'' of a generalized zero mean curvature surface and will use the abbreviation ZMC for ``zero mean curvature''.

Similarly to the classical minimal surface theory, a Weierstrass-type representation formula for ZMC surfaces in $\mathbb{R}^3(c)$ is stated as follows.

\begin{proposition}\label{prop:W}
Let $Fdw$ be a non-zero holomorphic 1-form on $\Sigma$ and $G$ a meromorphic function on $\Sigma$ such that $c|G|^2\not \equiv -1$ and $G^2Fdw$ is holomorphic. Assume that the holomorphic 1-forms
\begin{equation}\label{eq:1form}
\alpha_1=(1-cG^2)Fdw, \quad \alpha_2=-i(1+cG^2)Fdw, \quad \alpha_3=2GFdw
\end{equation}
on $\Sigma$ have no real periods. Then the mapping 
\begin{equation}\label{eq:w-formula}
X=X(c)=\mathrm{Re}\int {}^t(\alpha_1,\alpha_2,\alpha_3)
\end{equation}
gives a ZMC surface in $\mathbb{R}^3(c)$. Conversely, any ZMC surface in $\mathbb{R}^3(c)$ is of the form \eqref{eq:w-formula} provided that the surface is not part of the horizontal plane $t=\text{constant}$.
\end{proposition}

We call the pair $(F, G)$ the {\it Weierstrass data} of $X(c)$. By Proposition \ref{prop:W}, we see that the surfaces $\{X(c)\}_{c\in \mathbb{R}}$ share the same Weierstrass data $(F, G)$ unless $c$ satisfies $c|G|^2 \equiv -1$, which occurs only when $G$ is constant. In particular, when $c=1$ the formula \eqref{eq:w-formula} is nothing but the representation formula for minimal surfaces admitting branch points, when $c=-1$ the formula \eqref{eq:w-formula} is the representation for spacelike maximal surfaces with singularities derived by Estudillo-Romero \cite{ER}, Kobayashi \cite{K2} and Umehara-Yamada \cite{UY1}, and when $c=0$ the formula \eqref{eq:w-formula} is the representation for minimal surfaces in the isotropic 3-space $\mathbb{I}^3=\mathbb{R}^3(0)$, see \cite{SY,Pember,MaEtal,Sato,Si} (cf. \cite{Strubecker}) and their references.\\

From the next subsection, we give definitions of some deformations and transformations of surfaces in $\mathbb{R}^3(c)$ based on the classical minimal surface theory in $\mathbb{E}^3=\mathbb{R}^3(1)$.

\subsection{Associated family/Bonnet transformation}
The {\it associated family} $\{X_\theta(c) \}_{\theta \in  \mathbb{R}/2\pi \mathbb{Z}}$ of the surface $X(c)$ in $\mathbb{R}^3(c)$ is defined by the equation
\begin{equation}\label{eq:associated}
X_{\theta}=X_{\theta}(c)=\mathrm{Re}\bigint^w \left( \begin{array}{c} 1-cG^2\\ -i(1+cG^2)\\ 2G \end{array} \right) e^{i\theta}Fdw,
 \end{equation}
where $X=X_0$ is the original ZMC surface in $\mathbb{R}^3(c)$. The associated family of minimal surfaces was originally introduced by Bonnet \cite{Bonnet} and hence this bending  transformation from $X_0$ to $X_\theta$ is also called the {\it Bonnet transformation} and the parameter $\theta$ is sometimes called the {\it Bonnet angle}. In particular, $X_{
\frac{\pi}{2}}$ and $X_0$ are said to be the {\it conjugates} of each other. We denote $X_{\frac{\pi}{2}}$ by $X^*$.

The Bonnet transformation corresponds to changing the Weierstrass data from $(F,G)$ to $(e^{i\theta}F,G)$. Since the first fundamental form of \eqref{eq:associated} is written as $g_c=|F|^2(1+c|G|^2)^2dwd\overline{w}$, this deformation is an isometric deformation of the original one.

\subsection{L\'opez-Ros deformation/Goursat transformation} The {\it L\'opez-Ros deformation} $\{X_\lambda(c)\}_{\lambda >0}$ of $X(c)$ for each $c\in \mathbb{R}$ is a deformation changing the Weierstrass data of $X(c)$  
from $(F,G)$ to $(\frac{1}{\lambda}F,\lambda G)$ for $\lambda>0$. This deformation was introduced in \cite{LR} for minimal surfaces in $\mathbb{E}^3$, and in this case, the transformation $X_1(1)$ to $X_\lambda(1)$ is nothing but the Goursat transformation of the minimal surface $X_1(1)$:
\begin{equation*}
 X_\lambda(1)=\mathrm{Re}\bigint^w \left( \begin{array}{c} 1-\lambda^2 G^2 \\ -i\left(1+\lambda^2 G^2\right) \\ 2\lambda G \end{array} \right) \cfrac{F}{\lambda}dw,
\end{equation*}
which was originally introduced by Goursat \cite{Goursat} (see also \cite[p.120]{DHS}).

\subsection{$c$-deformation in $\mathbb{R}^3(c)$}
Changing the parameter $c\in \mathbb{R}$, the formula \eqref{eq:w-formula} gives a deformation of ZMC surfaces passing through different ambient spaces $\mathbb{R}^3(c)$. We call the deformation $\{X(c)\}_{c\in \mathbb{R}}$ the {\it $c$-deformation}. 

Obviously, the formula \eqref{eq:w-formula} connects the minimal surface $X(1)$ in $\mathbb{E}^3$ and the maximal surface $X(-1)$ in $\mathbb{L}^3$. Moreover, the point-wise relation $X(0)=(X(c)+X(-c))/2$ means that the minimal surface $X(0)$ in the isotropic 3-space $\mathbb{I}^3$ appears as the intermediate between $X(c)$ and $X(-c)$ for every $c$.

Under the L\'opez-Ros deformation in the previous subsection,  the height function of $X(c)$, that is, the third coordinate function of the surface, is preserved and a curvature line (resp. an asymptotic line) remains being a curvature line (resp. an asymptotic line). Impressively, the $c$-deformation is also furnished with the same properties as follows: 

\begin{proposition}\label{prop:c-deform}
The $c$-deformation $\{X(c)\}_{c\in \mathbb{R}}$ preserves its height function, and a curvature line $($resp. an asymptotic line$)$ remains being a curvature line $($resp. an asymptotic line$)$ under the $c$-deformation for any $c\neq 0$. 
\end{proposition}
Here, we consider the notation of curvature lines and asymptotic lines of surfaces in $\mathbb{R}^3(c)$ only for $c \neq 0$ because $\mathbb{R}^3(0)$ is not even a pseudo-Riemannian manifold.

\begin{remark} \label{rmk:isometry}
When we fix the sign of the parameter $c$ such as $c>0$, the composition of the surface $X(c)$ in $\mathbb{R}^3(c)$ with the Weierstrass data $(F,G)$ and the isometry
\[
\Phi(c)\colon \mathbb{R}^3(c) \longrightarrow \mathbb{R}^3(1)=\mathbb{E}^3,\quad {}^t(x,y,t)\longmapsto {}^t(x,y,\sqrt{c}t)
\] 
coincides with the surface $X(1)$ in $\mathbb{E}^3$ with the Weierstrass data $(F,\sqrt{c}G)$. A similar normalization is also valid for the case $c<0$. 
On the other hand, up to homotheties, the L\'opez-Ros deformation corresponds to the changing of the Weierstrass data from $(F,G)$ to $(F,\lambda G), \lambda>0$. This is the reason why the $c$-deformation and the L\'opez-Ros deformation share the same properties as in Proposition \ref{prop:c-deform}. However, changing the sign of $c$ in the $c$-deformation, which corresponds to changing ambient spaces will play an essential role in this paper. 
\end{remark}

\subsection{A classical  duality between surfaces in $\mathbb{E}^3$ and $\mathbb{L}^3$}
A ZMC surface $X(c)$ in $\mathbb{R}^3(c)$ is determined by the triplet of holomorphic 1-forms $(\alpha_1, \alpha_2, \alpha_3)$ in \eqref{eq:1form}. By \eqref{eq:conformal_regular}, the transformation
\[
(\alpha_1,\alpha_2,\alpha_3)\to (\alpha_1,\alpha_2,i\alpha_3)
\]
gives a ZMC surface in $\mathbb{R}^3(-c)$ unless $|\alpha_1|^2+|\alpha_2|^2-c|\alpha_3|^2\not \equiv 0$. We call this surface the \textit{dual} of $X(c)$, and denote it by $X^d(c)$.
For the case $c=1$, this transformation gives the classical duality of minimal surfaces in $\mathbb{E}^3$ and maximal surfaces in $\mathbb{L}^3$, discussed in many literatures, for example see \cite{LLS, AL, UY1, Lee}. As known in \cite{Lee} (see also \cite[Proposition 2.2]{AF}), it is a global version of the duality which was used by Calabi \cite{C} to prove the Bernstein theorem for maximal surfaces in $\mathbb{L}^3$.
Notably, the surfaces $X(c)$ in $\mathbb{R}^3(c)$ and $X(-c)$ in $\mathbb{R}^3(-c)$ are related by the above duality as follows.

\begin{proposition} \label{prop:dual_vs_conj}
ZMC surfaces $X(c)$ and $X(-c)$ defined by \eqref{eq:1form} and \eqref{eq:w-formula} are related by the equation 
\[
X^d(c)=J\circ X^*(-c),
\]
where $J$ is the counterclockwise rotation by angle $\pi/2$ in the $xy$-plane.
\end{proposition}

\begin{proof}
Let us denote $X(c)$ in \eqref{eq:w-formula} by ${}^t(\Re (\psi_1),\Re (\psi_2), \Re (\psi_3))$.  The dual $X^d(c)$ of $X(c)$ is written as ${}^t(\mathrm{Re}(\psi_1),\mathrm{Re}(\psi_2),-\mathrm{Im}(\psi_3))$. On the other hand, by  \eqref{eq:1form} and \eqref{eq:w-formula}, $X^*(-c)$ is written as
\begin{align*}
X^*(-c)&=\mathrm{Re}\bigint^w \left( \begin{array}{c} 1+cG^2 \\ -i(1-cG^2) \\ 2G \end{array} \right) iFdw \\
&=\mathrm{Re}\bigint^w \left( \begin{array}{c} i(1+cG^2) \\ 1-cG^2 \\ 2iG \end{array} \right)Fdw = {}^t(-\mathrm{Re}(\psi_2),\mathrm{Re}(\psi_1),-\mathrm{Im}(\psi_3)).
\end{align*}
Hence, we obtain the desired relation $J\circ X^*(-c)=X^d(c)$.
\end{proof}

\section{Deformation Family and its Graphness} \label{sec:def_family}
First, we give a unified form of the three types of deformations in the previous section.

\subsection{Deformation family} \label{subsec:def_family}
For a given Weierstrass data $(F,G)$, let us consider the three parameter family of surfaces $X_{\theta, \lambda, c}\colon \Sigma \to \Rc$ defined by
\begin{equation} \label{eq:extended_family}
	X_{\theta, \lambda, c}=\Re \bigints^w \left( \begin{array}{c} 1-c\lambda^2 G^2 \\ -i(1+c\lambda^2 G^2) \\ 2\lambda G \end{array} \right) \frac{e^{i\theta}}{\lambda}Fdw.
\end{equation}
Here, $\theta \in \RZ$ is the Bonnet angle, $\lambda\in(0,+\infty)$, and $c\in\mathbb{R}$ as in Section \ref{sec:2}. By its definition, $X_{\theta, \lambda, c}$ is a ZMC surface in $\Rc$ in the sense of Section \ref{subsec:w_rep_in_Rc}.

Let $\mathcal{P} = \RZ \times (0, +\infty) \times \mathbb{R}$ be the parameter space of $X_{\theta,\lambda,c}$.
The three parameter family $\{X_{\theta,\lambda,c}\}_{\theta,\lambda,c}$ unifies the deformations explained in the previous sections (see also Fig. \ref{fig:def_family_intro}).

\subsection{Graphness} \label{subsec:graphness}
Let $(x,y,t)\in \Rc$ be the canonical coordinate.  We say that $S_{\theta, \lambda, c}:=X_{\theta,\lambda,c}(\Sigma)\subset \Rc$ is a \textit{graph} if there is a domain $\Omega=\Omega_{\theta, \lambda, c} \subset xy$-plane and a function $\phi=\phi_{\theta, \lambda, c} \colon \Omega \to \mathbb{R}$ such that $S_{\theta, \lambda, c}={\rm graph}(\phi)=\{(x, y , \phi(x,y)) \mid (x,y)\in \Omega\}$.

Hereafter, we discuss the \textit{graphness} of $S_{\theta,\lambda, c}$, that is, the property that $S_{\theta,\lambda,c}$ is whether a graph or not. In particular, we consider the case where the Riemann surface $\Sigma$ which is a domain of $X_{\theta,\lambda,c}$ is a simply connected proper subdomain $D\subset \mathbb{C}$.\\

Under the identification that $xy$-plane $\cong \mathbb{C}$, $(x,y)\mapsto x+iy$, we assume that $\Rc \cong \mathbb{C}\times \mathbb{R}$. Elementary calculations show the following lemma and proposition that connect the surface theory with the planar harmonic mapping theory. In particular, the univalent harmonic mapping theory is directly applicable to the problems on the graphness of $S_{\theta,r,c}$.

\begin{lemma}
For a given Weierstrass data $(F,G)$, let
\begin{eqnarray*}
	h=\int^w F dw,\ \ g=-\int^w G^2F dw, \ \ T=\int^w 2GFdw.
\end{eqnarray*}
Further, we put $t=\Re(T)$. Then
\begin{equation}\label{eq:planar_rep_X}
	X_{\theta, \lambda, c}=\left( \begin{array}{c} \displaystyle \frac{e^{i\theta}}{\lambda}h+c\lambda e^{-i\theta}\overline{g} \\[10pt] \Re(e^{i\theta}T) \end{array} \right)=\left( \begin{array}{c} \displaystyle \frac{e^{i\theta}}{\lambda}\left( h+c\lambda^2e^{-2i\theta}\overline{g}\right) \\[10pt] t\cos\theta - t^{\ast}\sin\theta \end{array} \right),
\end{equation}
where $t^{\ast}$ denotes the conjugate harmonic function of $t$.
\end{lemma}

\noindent
The graphness of $S_{\theta,\lambda, c}$ is characterized by the univalence of the following planar harmonic mapping.
\begin{proposition}\label{prop:graphness_vs_univalence}
Under the above notations, let us define
\[
	f_{\theta, \lambda, c}=h+c\lambda^2e^{-2i\theta}\overline{g}.
\]
Then, $S_{\theta, \lambda, c}$ is a graph if and only if $f_{\theta,\lambda,c}$ is univalent.
\end{proposition}

In general, a planar harmonic mapping $f\colon D \to \mathbb{C}$ can be decomposed uniquely (up to additive constants) into the form $f=h+\overline{g}$ by using holomorphic functions $h$ and $g$. The meromorphic function $\omega=\omega_f:=\overline{f_{\overline{w}}}/f_w=g'/h'$ is called the {\it analytic dilatation} (or the {\it second Beltrami coefficient}) of $f$. This analytic dilatation is one of the most important quantities in the theory of planar harmonic mappings (for details, see \cite{D}). In particular, if $f$ is univalent and orientation-preserving, then $\omega$ is holomorphic and $|\omega|<1$ on $D$ since the Jacobian of $f$ satisfies $Jf=|f_w|^2-|f_{\overline{w}}|^2=|h'|^2-|g'|^2>0$.  The next formula indicates that the analytic dilatation of $f_{\theta, \lambda, c}$ corresponds to $G$ of the Weierstrass data, and  plays an important role in the later sections.

\begin{lemma}\label{lem:dilatation}
Let $\omega_{\theta, \lambda, c}$ be the analytic dilatation of $f_{\theta, \lambda, c}$. Then,
\begin{equation}
	\omega_{\theta, \lambda, c} = -c\lambda^2e^{2i\theta}G^2.
\end{equation}
\end{lemma}

\begin{remark}\label{rmk:necess_cond}
Hereafter, we usually assume $|G|<1$ for Weierstrass data of deformation families. This condition is supposed to be a necessary condition for $S_{0, 1, \pm 1}$ to be a graph under a certain normalization: Assume that $S_{0,1,\pm 1}$ is a graph. Then $f:=f_{0,1,\pm 1}$ is univalent and thus Lewy's theorem implies that its Jacobian satisfies $Jf>0$ or $Jf<0$. Here, Lewy's theorem states that a planar harmonic mapping is locally univalent at some point if and only if its Jacobian does not vanish at that point (see, \cite[Section 2.2]{D}). On the other hand, it holds that $Jf=|f_{w}|^2-|f_{\overline{w}}|^2=|h'|^2(1-|G|^4)$ by Lemma \ref{lem:dilatation}. Therefore, we have $|G|<1$ (if $Jf>0$) or $|G|>1$ (if $Jf<0$). When $|G|>1$, by commuting the $x$-axis and the $y$-axis, we can normalize the situation to the case of $|G|<1$.

We emphasize that the condition $|G|<1$ is not a sufficient condition of surfaces to be graphs. For example, if we consider a minimal surface in $\mathbb{E}^3$ with the Weierstrass data $(F,G)$, then $G$ corresponds to the Gauss map of $X_{0,1,1}$ via the stereographic projection and $|G|<1$ merely means that the image of the Gauss map is contained in a hemisphere.

\end{remark}

\section{Krust Type Theorem Part 1} \label{sec:Krust_1}
In this and the next sections, we discuss conditions for $S_{\theta,\lambda,c}$ to be a graph. As for the theorems which deal with the graphness of surfaces, it is well-known as the Krust theorem that if a minimal (resp. maximal) surface is a graph over a convex domain, then each member of its associated family is also a graph (over a close-to-convex domain, see \cite[Corollary 3.4]{Dorff}). Further, the Krust-type theorem for the L\'opez-Ros deformation is also obtained by Dorff in \cite[Corollary 3.5]{Dorff}. In this section, we prove a Krust-type theorem for the family $\{S_{\theta,\lambda,c}\}$ including the above theorems.\\

We first recall some notions of convexities. A domain $\Omega \subset \mathbb{C}$ is said to be a \textit{convex domain} if any two points $z_1, z_2 \in \Omega$ are connected by a segment in $\Omega$, that is, $[z_1,z_2]\subset \Omega$. Further, $\Omega$ is called a \textit{close-to-convex domain} if its compliment $\mathbb{C}\setminus \Omega$ can be represented by a union of half lines that are disjoint except possibly for their initial points.

In addition to these concepts, there are various kinds of convexity for domains. For example, $\Omega$ is said to be a \textit{starlike domain} with respect to a point $a\in \Omega$ if the segment $[z,a]$ is contained in $\Omega$ for any $z\in \Omega$. The (ordinary) convexity implies the starlike convexity, and the starlike convexity implies the close-to-convexity in general. Thus, $\Omega$ is convex, then it is close-to-convex. For details of these convexities, see \cite{Pommerenke}.
 
 We say that a planar harmonic mapping $f\colon D \to \mathbb{C}$ is \textit{convex} (resp. \textit{close-to-convex}) if $f$ is univalent and $f(D)$ is convex (resp. close-to-convex). To obtain our result, we use the following theorem (see also \cite{CS84}).

\begin{theorem}[Kalaj, {\cite[Theorem 2.1]{Kalaj}}]\label{thm:kalaj}
Let $f=h+\overline{g}\colon \mathbb{D}\to\mathbb{C}$ be a univalent orientation-preserving harmonic mapping, where $\mathbb{D}$ denotes the unit disk $\mathbb{D}=\{w\in\mathbb{C}\mid |w|<1\}$. If $f$ is convex, then $f_{\varepsilon}:=h+\varepsilon \overline{g}$ is close-to-convex and orientation-preserving for all $\varepsilon \in \overline{\mathbb{D}}$.
\end{theorem}

\begin{remark}\label{rmk:generality_domain}
On the above Theorem \ref{thm:kalaj}, the domain of definition of $f$ is not needed to be $\mathbb{D}$. Suppose $D$ is a simply connected proper subdomain of $\mathbb{C}$ and $f=h+\overline{g}\colon D\to \mathbb{C}$ is a convex orientation-preserving harmonic mapping. Taking a conformal mapping $\Phi \colon \mathbb{D} \to D$, we define $\mathcal{F}:=f\circ \Phi=h\circ \Phi +\overline{g\circ \Phi}$. Then Theorem \ref{thm:kalaj} can be applied to $\mathcal{F}$, and thus $\mathcal{F}_{\varepsilon}=h\circ \Phi + \varepsilon \overline{g\circ \Phi}$ is close-to-convex ($\varepsilon\in \overline{\mathbb{D}}$). Therefore, $f_{\varepsilon}:=\mathcal{F}_{\varepsilon}\circ \Phi^{-1}=h+\varepsilon\overline{g}$ is close-to-convex.
\end{remark}

\begin{remark}\label{rmk:sense_preserving}
Contrary to Remark \ref{rmk:generality_domain}, the assumption that $f$ is orientation-preserving is important. In fact, if $f=h+\overline{g}$ is a convex orientation-reversing harmonic mapping, then the corresponding conclusion is that $\widetilde{f_{\varepsilon}}:=\varepsilon h +\overline{g}$ is close-to-convex for all $\varepsilon \in \overline{\mathbb{D}}$.
\end{remark}

Let us consider a deformation family $S_{\theta,\lambda,c}=X_{\theta,\lambda,c}(D)$ with the Weierstrass data $(F,G)$ in \eqref{eq:extended_family}. The following theorem holds.
\begin{theorem}\label{thm:main_krust1}
Suppose that $G$ is not constant and $|G|<1$. If there exists $(\theta_0,\lambda_0,c_0)\in \mathcal{P}$ such that $|c_0 \lambda_0^2|\leq 1/\| G \|_{\infty}^2$ and $S_{\theta_0,\lambda_0,c_0}$ is a graph over a convex domain, then $S_{\theta,\lambda,c}$ is a graph over a close-to-convex domain for any $(\theta,\lambda,c)\in\mathcal{P}$ with $|c\lambda^2|\leq |c_0 \lambda_0^2|$. Here, $\| G \|_{\infty}:=\sup_{w\in D}|G(w)|$.
\end{theorem}

\begin{proof}
By the assumption, $f_0:=f_{\theta_0,\lambda_0,c_0}=h+c_0 \lambda_0^2 e^{-2i\theta_0} \overline{g}=h_0+\overline{g_0},\ (h_0:=h,\ g_0:=c_0 \lambda_0^2 e^{2i\theta_0}g)$ is convex. Since $|c_0 \lambda_0^2| \leq1/\| G \|_{\infty}^2$, Lemma \ref{lem:dilatation} and the maximum principle imply that $|\omega_{f_0}|=|c_0 \lambda_0^2 G^2|<1$. Thus $f_0$ is orientation-preserving. Therefore, $h_0+\varepsilon \overline{g_0}$ is close-to-convex for $\varepsilon \in \overline{\mathbb{D}}$ by Theorem \ref{thm:kalaj}.

On the other hand, if $(\theta,\lambda,c)\in\mathcal{P}$ satisfies $|c\lambda^2|\leq |c_0 \lambda_0^2|$ and $c_0\neq 0$, then
\begin{equation*}
	f_{\theta,\lambda,c}=h+c\lambda^2 e^{-2i\theta}\overline{g}=h_0+\varepsilon \overline{g_0},\ \ \ \varepsilon:=\frac{c\lambda^2}{c_0 \lambda_0^2} e^{-2i(\theta-\theta_0)} \in \overline{\mathbb{D}}.
\end{equation*}
Thus, $f_{\theta,\lambda,c}$ is close-to-convex. When $c_0= 0$, $|c\lambda^2|\leq |c_0 \lambda_0^2|$ implies $c=0$ and hence $f_{\theta,\lambda,c}=h=f_{\theta_0,\lambda_0,c_0}$ is also close-to-convex. This is the desired conclusion.
\end{proof}

\begin{figure}[htbp]
       \begin{center}
           \includegraphics[scale=0.75]{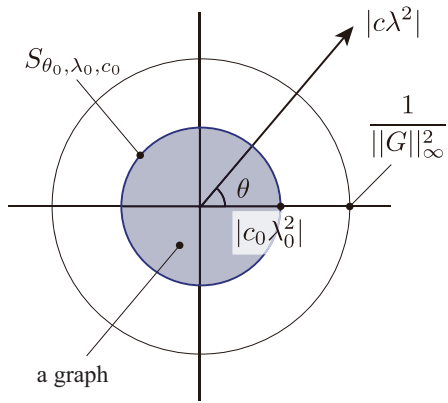} 
       \end{center}
       \caption{The highlighted part indicates the region on which the ZMC surfaces $S_{\theta,\lambda,c}$ are graphs in the case where $S_{\theta_0,\lambda_0,c_0}$ is a graph over a convex domain. The parameter space is described in polar coordinates with radius $|c\lambda^2|$ and angle $\theta$. }  
\end{figure} 

Putting $(\theta_0,\lambda_0,c_0)=(0,1,\pm 1)$ in Theorem \ref{thm:main_krust1}, we obtain the following.
\begin{corollary}\label{cor:main_krust1}
Suppose that $G$ is not constant and $|G|<1$. 
If $S_{\min}:=S_{0,1,1}$ $($or $S_{\max}:=S_{0,1,-1})$ is a graph over a convex domain, then
\vspace*{0.5em}
\begin{enumerate}\setlength{\parskip}{0.5em}
	\item[$(\mathrm{i})$] for any $(\theta,\lambda,c)\in \mathcal{P}$ with $|c\lambda^2|\leq 1$, $S_{\theta,\lambda,c}$ is a graph over a close-to-convex domain.
\end{enumerate}
\vspace*{0.5em}
In particular, the following hold.
\vspace*{0.5em}
\begin{enumerate}\setlength{\parskip}{0.5em}
	\item[$(\mathrm{ii})$] Classical Krust's theorem: For any $\theta\in \RZ$, $S_{\theta,1,\pm 1}$ is a graph over a close-to-convex domain.
	\item[$(\mathrm{iii})$] Krust-type theorem for L\'opez-Ros deformation: For any $\lambda\in (0,1]$, $S_{0,\lambda,\pm 1}$ is a graph over a close-to-convex domain.
\end{enumerate}
\vspace*{0.5em}
\end{corollary}

\begin{remark}\label{rmk:main_cor}
It should be remarked that the above assertion (i) is not just a combination of the assertions (ii) and (iii). In fact, under the assumptions in Corollary \ref{cor:main_krust1},
\vspace*{0.5em}
\begin{itemize}\setlength{\parskip}{0.5em}
	\item $S_{\theta,1,1}$ is not necessarily a graph over a convex domain, even if $S_{\min}=S_{0,1,1}$ is a graph over a convex domain. However (i) shows that its L\'opez-Ros deformation $S_{\theta,\lambda,1}\ (\lambda\leq 1)$ is also a graph.
	\item $S_{0,\lambda,1}\ (\lambda\leq 1)$ is not necessarily a graph over a convex domain, even if $S_{\min}=S_{0,1,1}$ is a graph over a convex domain. However (i) shows that its associated family consists of graphs.
\end{itemize}
\vspace*{0.5em}
\end{remark}

\section{Krust Type Theorem Part 2} \label{sec:Krust_2}
In this section, we prove another kind of Krust-type theorem, which does not assume the convexity of minimal surfaces in $\mathbb{E}^3$. This result indicates that it is powerful to consider deformations of minimal (resp. maximal) surfaces in the Euclidean space (resp. Minkowski space) across different spaces $\Rc$. In other words, the parameter $c$ in the deformation family $\{S_{\theta,\lambda,c}\}$ plays a significant role in the discussions.

\subsection{Another kind of Krust type theorem}\label{subsec:new_krust}
First, we give a simple observation. By the equation \eqref{eq:planar_rep_X} and the definition of $f_{\theta,\lambda,c}$, we immediately have
\begin{equation*}
	X_{\theta, \lambda, 0}=\left( \begin{array}{c} \displaystyle \frac{e^{i\theta}}{\lambda} h \\[10pt] t\cos\theta - t^{\ast}\sin\theta \end{array} \right),\ \ \ \ f_{\theta,\lambda,0}=h.
\end{equation*}
\noindent
Thus, by using Proposition \ref{prop:graphness_vs_univalence}, we have the following.

\begin{proposition}
The following are equivalent:
\vspace*{0.5em}
\begin{enumerate}\setlength{\parskip}{0.5em}
	\item[$(\mathrm{i})$] $S_{\theta_0,\lambda_0,0}$ is a graph $($over a convex domain$)$ for some $(\theta_0,\lambda_0)$.
	\item[$(\mathrm{ii})$] $S_{\theta,\lambda,0}$ is a graph $($over a convex domain$)$ for any $(\theta,\lambda)$.
	\item[$(\mathrm{iii})$] $h$ is univalent $($and convex$)$.
\end{enumerate}
\vspace*{0.5em}
\end{proposition}

The next result that we will use is obtained by Partyka-Sakan in \cite[Theorem 3.1]{PS} (see also \cite[Theorem 5.17]{CS84}).

\begin{theorem}[Partyka-Sakan, \cite{PS}]\label{thm:PS}
Let $h, g\colon \mathbb{D}\to \mathbb{C}$ be holomorphic functions such that $|h'|>|g'|$ and $\omega=g'/h'$ is not constant. If there exists $\varepsilon_0 \in \mathbb{C}$ such that $|\varepsilon_0|\| \omega \|_{\infty} \leq 1$ and $h+\varepsilon_0 g$ is convex $($i.e. it is a conformal mapping and its image is convex$)$, then $f_{\varepsilon}:=h+\varepsilon \overline{g}$ is a close-to-convex harmonic mapping for any $\varepsilon \in \mathbb{C}$ with $|\varepsilon|\| \omega \|_{\infty}\leq 1$.
\end{theorem}

\begin{remark}
If $h$ itself is a convex function, the assumption of Theorem \ref{thm:PS} is satisfied for $\varepsilon_0=0$.
\end{remark}

\begin{remark}
Theorem \ref{thm:PS} is still valid if we use a general simply connected proper subdomain $D\subset \mathbb{C}$ as the domain of $h$ and $g$ instead of $\mathbb{D}$, for the same reason as for Theorem \ref{thm:kalaj} (see Remark \ref{rmk:generality_domain}).
\end{remark}

\begin{theorem}\label{thm:main_krust2}
Suppose that $G$ is not constant and $|G|<1$. If the minimal surface $S_{\theta_0,\lambda_0,0}$ in the isotropic $3$-space $\mathbb{I}^3$ is a graph over a convex domain for some $(\theta_0,\lambda_0)$, then $S_{\theta,\lambda,c}$ is a graph over a close-to-convex domain for any $(\theta,\lambda,c)\in\mathcal{P}$ with $|c\lambda^2|\leq 1/\| G \|_{\infty}^2$. In particular, the minimal and the maximal surfaces $S_{\theta, 1, \pm 1}$ are graphs.
\end{theorem}

\begin{proof}
By Lemma \ref{lem:dilatation}, it holds that
\[
	-G^2 = \omega_{0,1,1}=\frac{g'}{h'}.
\]
Thus $|h'|>|g'|$ and $g'/h'$ is not constant since $|G|<1$ and $G$ is not constant. Further, $h$ is a convex conformal mapping by the assumption that $S_{\theta_0,\lambda_0,0}$ is a graph over a convex domain. Theorem \ref{thm:PS} implies that $h+\varepsilon\overline{g}$ is a close-to-convex harmonic mapping for any $\varepsilon\in\mathbb{C}$ with $|\varepsilon|\|g'/h'\|_{\infty}=|\varepsilon|\| G \|_{\infty}^2\leq 1$. Therefore, for any $(\theta,\lambda,c)\in\mathcal{P}$ with $|c\lambda^2|\| G \|_{\infty}^2\leq 1$, the planar harmonic mapping $f_{\theta,\lambda,c}=h+c\lambda^2e^{-2i\theta}\overline{g}$ is close-to-convex.
\end{proof}

\begin{figure}[htbp]
       \begin{center}
           \includegraphics[scale=0.75]{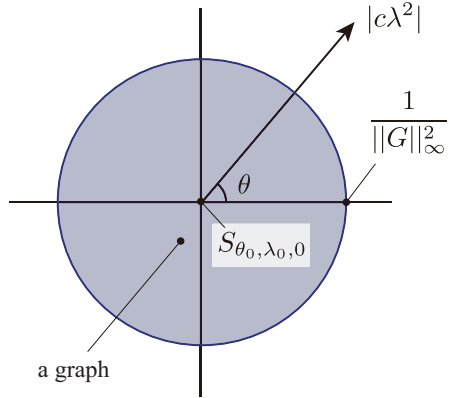} 
       \end{center}
       \caption{The region on which $S_{\theta,\lambda, c}$ is a graph in the case where the ZMC surface $S_{\theta_0,\lambda_0,0}$ is a graph over a convex domain.}  
\end{figure}

Recall that $h$ is completely determined by the holomorphic data $F$ of the Weierstrass data $(F,G)$. The following is just a restatement of Theorem \ref{thm:main_krust2} in terms of the Weierstrass data.
\begin{corollary}\label{cor:rephrasing_Krust2}
Assume that the Weierstrass data $(F,G)$ satisfies that 
\vspace*{0.5em}
\begin{enumerate}\setlength{\parskip}{0.5em}
	\item[$(\mathrm{i})$] $G$ is not constant and $|G|<1$,
	\item[$(\mathrm{ii})$] $\displaystyle h:=\int^w F dw$ is convex.
\end{enumerate}

\vspace*{0.5em}
\noindent Then, $S_{\theta,\lambda,c}$ is a graph over a close-to-convex domain for any $(\theta,\lambda,c)\in\mathcal{P}$ with $|c\lambda^2|\leq 1/\| G \|_{\infty}^2$.
\end{corollary}

\subsection{Case of $D=\mathbb{D}$}
Here, we discuss a particular case where the domain $D$ is the unit disk $\mathbb{D}=\{|w|<1\}$. For an arbitrary $0<R\leq 1$, set $\mathbb{D}_R:=\{|w|<R\}$ and
\begin{equation*}
	S_{\theta,\lambda,c}^R :=X_{\theta,\lambda,c}(\mathbb{D}_R),\ \ \ \ G^R:=G|_{\mathbb{D}_R}.
\end{equation*}

The Krust-type theorem for L\'opez-Ros deformation with the parameter $\lambda$ of minimal surfaces in $\mathbb{E}^3$ proved in \cite{Dorff} is valid only for $0<\lambda \leq 1$. On the other hand, Theorem \ref{thm:main_krust2} for deformation family with parameters $(\theta,\lambda,c)\in \mathcal{P}$ is more broadly valid for $|c\lambda^2|\leq 1/\| G \|_\infty^2$, where $1/\| G \|_\infty^2\geq 1$, and this assumption is optimal as we will see in Section \ref{subsec:non_graph}. Even in the case $|c\lambda^2|> 1/\| G \|_\infty^2$, we can prove the following Krust-type theorem by restricting domains of graphs.

\begin{corollary}\label{cor:subgraph1}
Suppose $D=\mathbb{D}$, $G$ is not constant, and $|G|<1$. If the minimal surface $S_{\theta_0,\lambda_0,0}$ in $\mathbb{I}^3$ is a graph over a convex domain for some $(\theta_0,\lambda_0)$, then $S_{\theta,\lambda,c}^R$ is a graph over a close-to-convex domain for any $(\theta,\lambda,c)\in \mathcal{P}$ with $|c\lambda^2|\leq 1/\|G^R\|_{\infty}^2$.
\end{corollary}

\begin{proof}
By the assumption, $h$ is a convex conformal mapping. Thus, the restriction $h^R:=h|_{\mathbb{D}_R}$ is also a convex conformal mapping. This assertion immediately follows from the well-known fact that a locally univalent holomorphic function $\varphi\colon \mathbb{D}\to \mathbb{C}$ is convex if and only if $\Re \left( 1+ w \varphi''(w)/\varphi'(w)\right)>0$ holds for every $w\in \mathbb{D}$ (see, \cite[Section 3.6]{Pommerenke}). Therefore, we obtain the conclusion by applying Theorem \ref{thm:main_krust2} to the restricted deformation family $\{S_{\theta,\lambda,c}^R\}$.
\end{proof}

\begin{corollary}\label{cor:subgraph2}
If we suppose $G(0)=0$ in addition to the same assumptions as Corollary \ref{cor:subgraph1}, then the same conclusion as the corollary holds for $(\theta,\lambda,c)\in \mathcal{P}$ with $|c\lambda^2|\leq 1/R^2$.
\end{corollary}

\begin{proof}
The assumptions and the Schwarz lemma shows that $|G(w)|\leq |w|$. Thus we have $\|G^R\|_{\infty}\leq R$, and Corollary \ref{cor:subgraph1} implies the desired assertion.
\end{proof}

We again recall that the mapping $h$ is defined by only $F$. According to the analytic characterization of convex conformal mappings (see the proof of Corollary \ref{cor:subgraph1}), we can translate the assumption that $S_{\theta_0, \lambda_0, 0}$ is a graph over a convex domain into an analytic condition of the Weierstrass data as follows.
\begin{corollary}\label{cor:rephrasing_Krust2_for_disk}
Suppose $D=\mathbb{D}$. If the Weierstrass data $(F,G)$ satisfies that 
\vspace*{0.5em}
\begin{enumerate}\setlength{\parskip}{0.5em}
	\item[$(\mathrm{i})$] $G$ is not constant and $|G|<1$,
	\item[$(\mathrm{ii})$] $F(w)\neq 0$ and $\displaystyle \Re \left( 1+ w \frac{F'(w)}{F(w)}\right)>0$ hold for every $w\in \mathbb{D}$,
\end{enumerate}
\vspace*{0.5em}
then $S_{\theta,\lambda,c}$ is a graph over a close-to-convex domain for any $(\theta,\lambda,c)\in\mathcal{P}$ with $|c\lambda^2|\leq 1/\| G \|_{\infty}^2$.
\end{corollary}

\subsection{Sharpness of estimations.}\label{subsec:non_graph}
At last, we show a condition for $S_{\theta,\lambda,c}=X_{\theta,\lambda,c}(D)$ not to be a graph, and discuss the sharpness of the estimations for the realms in which $S_{\theta,\lambda,c}$ is a graph in the Krust-type theorems that we obtained. 

\begin{theorem}\label{thm:non_graph}
Suppose that $G$ is not a constant function and $|G|<1$.  For any $(\theta, \lambda, c)\in \mathcal{P}$, if $1/\| G \|_{\infty}^2 < |c\lambda^2|<1/\inf |G|^2$ $($or $1/\| G \|_{\infty}^2<|c\lambda^2|<+\infty$ if $\inf |G|^2=0)$, then $S_{\theta, \lambda, c }$ is not a graph. 
\end{theorem}

\begin{proof}
If $(\theta,\lambda,c)\in \mathcal{P}$ satisfies the assumption, it holds that $|c\lambda^2| \inf|G|^2 < 1 < |c\lambda^2|  \| G \|_{\infty}^2$. Thus, there exists $w\in D$ such that $|c\lambda^2| |G(w)|^2=1$. By Lemma \ref{lem:dilatation}, we have $|\omega_{\theta,\lambda,c}(w)|=|c\lambda^2||G(w)|^2 = 1$, and this implies that the Jacobian $Jf_{\theta,\lambda,c}(w)=0$. Thus, $f_{\theta,\lambda,c}$ is not univalent by Lewy's theorem as explained in Remark \ref{rmk:necess_cond}.
\end{proof}

Theorem \ref{thm:main_krust2} claims that if the isotropic minimal surface $S_{\theta_0,\lambda_0,0}$ is a graph over a convex domain for some $(\theta_0,\lambda_0)$, then $S_{\theta,\lambda,c}$ is a graph for any $(\theta,\lambda,c)$ satisfying $|c\lambda^2|\leq 1/\| G \|_{\infty}^2$. According to the above simple observation, it is revealed that the estimation $|c\lambda^2|\leq 1/\| G \|_{\infty}^2$ in Theorem \ref{thm:main_krust2} is optimal, see Fig.  \ref{fig:Krust2_optimal}.

\begin{figure}[htbp]
       \begin{center}
           \includegraphics[scale=0.75]{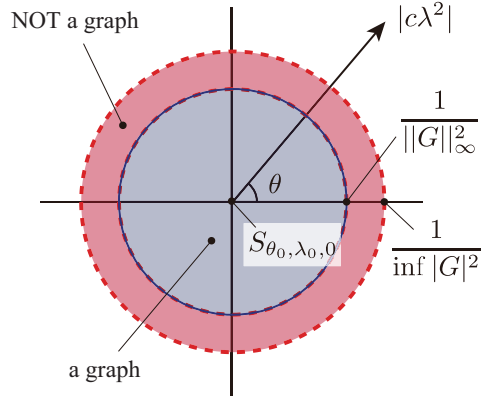} 
       \end{center}
       \caption{The estimation in Theorem \ref{thm:main_krust2} is optimal.}   \label{fig:Krust2_optimal}
\end{figure}
  
 In contrast, the estimation in Theorem \ref{thm:main_krust1} is not optimal in general, see the left side of Fig.  \ref{fig:Krust1_non_optimal}. However, for example, if $S_{\min}=S_{0,1,1}$ is a graph over a convex domain and $\| G \|_{\infty}=1$, then the estimation becomes optimal, see the right side of Fig.  \ref{fig:Krust1_non_optimal}.
  
\begin{figure}[htbp]
\hspace*{-5ex}
    \begin{tabular}{cc}
      \begin{minipage}[t]{0.55\hsize}
        \centering
        \includegraphics[keepaspectratio, scale=0.75]{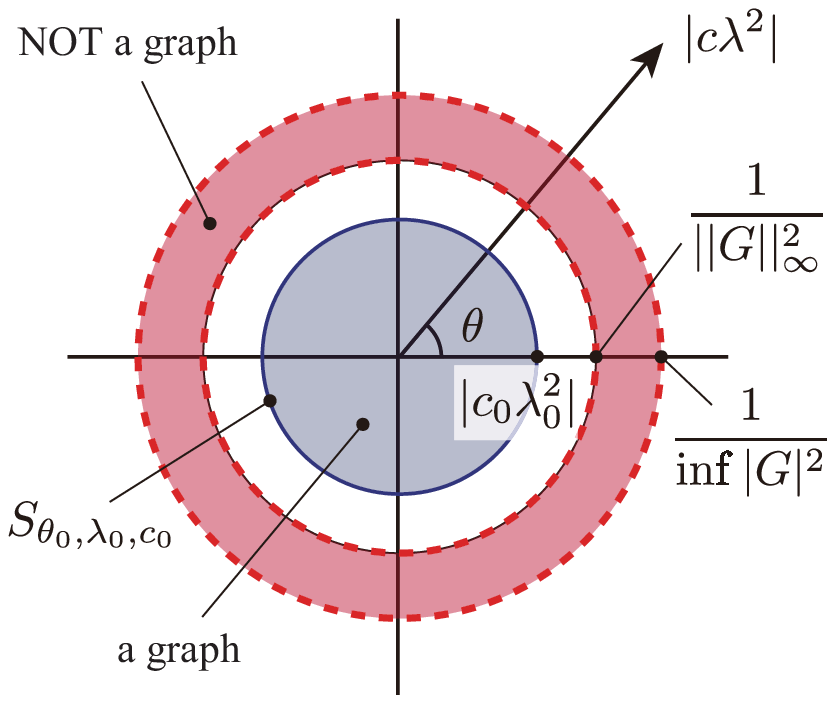}
      \end{minipage} &
      \hspace*{-5ex}
      \begin{minipage}[t]{0.55\hsize}
        \centering
        \includegraphics[keepaspectratio, scale=0.75]{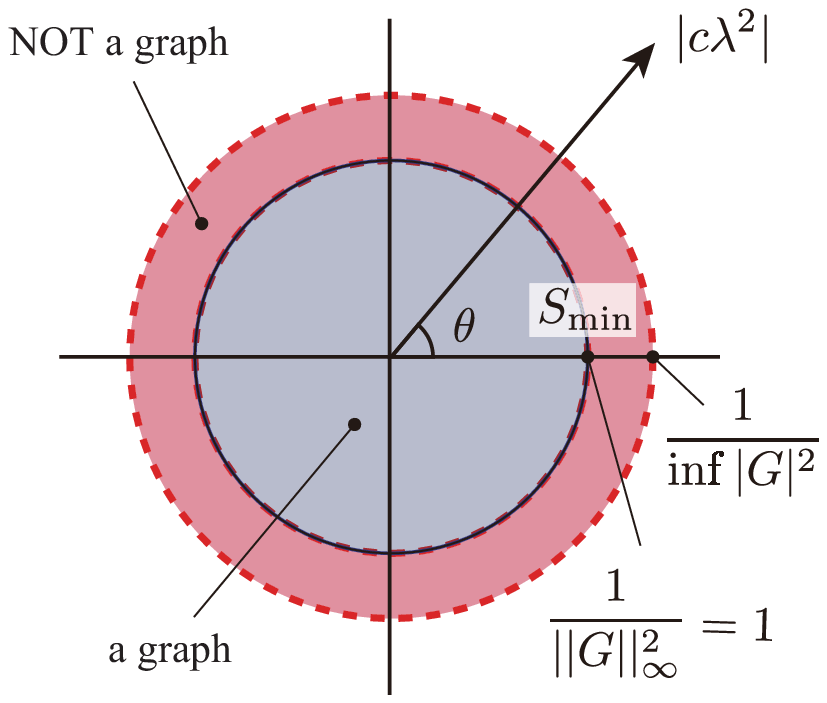}
      \end{minipage}
    \end{tabular}
    \caption{Left: the estimation in Theorem \ref{thm:main_krust1} is not optimal in general. Right: if $S_{\min}=S_{0,1,1}$ is a graph over a convex domain and $\| G \|_{\infty}=1$, then the estimation in Theorem \ref{thm:main_krust1} becomes optimal.}\label{fig:Krust1_non_optimal}
\end{figure}

\subsection{Generalization to graphs over non-convex domains.}
When we look at the minimal surface $X_{\theta, \lambda, 0}$ in $\mathbb{I}^3$ of the deformation family $X_{\theta, \lambda, c}$, there is a further generalization of Theorem \ref{thm:main_krust2} with slightly worse estimation. In this case, surprisingly, we can weaken the convexity assumption for the minimal graph in $\mathbb{I}^3$. To obtain the result, we use the next theorem (see also \cite{CH}).
\begin{theorem}[Partyka-Sakan-Zhu, {\cite[Theorem 2.8]{PSZ18}}]\label{thm:PSZ}
Let $h, g\colon \mathbb{D}\to \mathbb{C}$ be holomorphic functions such that $|h'|>|g'|$. If $h$ is  univalent and its image $h(\mathbb{D})$ is a rectifiably $M$-arcwise connected domain for some $M\geq 1$, then $f_{\varepsilon}:=h+\varepsilon \overline{g}$ is a univalent (more strongly, quasiconformal) harmonic mapping for any $\varepsilon \in \mathbb{C}$ with $|\varepsilon|\| \omega \|_{\infty}< 1/M$, where $\omega:=g'/h'$.
\end{theorem}
\noindent
Here, for any $M\geq 1$, a planar domain $\Omega\subset\mathbb{C}$ is said to be \textit{rectifiably $M$-arcwise connected} (or \textit{linearly connected} with constant $M$) if for all $z, w\in \Omega$ there exists a rectifiable arc $\gamma$ in $\Omega$ which joins $z$ and $w$ such that ${\rm Length}(\gamma)\leq M|z-w|$. By definition, a domain $\Omega$ is rectifiably $1$-arcwise connected if and only if it is convex. However, rectifiably $M$-arcwise connected domains are not necessarily convex, starlike or close-to-convex if $M>1$. By this theorem, the following result can be proved in the same way as Theorem \ref{thm:main_krust2}.
\begin{theorem}\label{thm:main_krust3}
Suppose that $G$ is not constant and $|G|<1$. If the minimal surface $S_{\theta_0,\lambda_0,0}$ in the isotropic $3$-space $\mathbb{I}^3$ is a graph over a rectifiably $M$-arcwise connected domain for some $M\geq 1$ and $(\theta_0,\lambda_0)$, then $S_{\theta,\lambda,c}$ is a graph for any $(\theta,\lambda,c)\in\mathcal{P}$ with $|c\lambda^2|< 1/(M\| G \|_{\infty}^2)$.
\end{theorem}
It should be remarked that if $M=1$ and $\|G\|_{\infty}=1$ in Theorem \ref{thm:main_krust3}, then the estimation becomes $|c\lambda^2|<1$. Therefore we cannot see whether the minimal surfaces $S_{\theta, 1 ,1}$ and the maximal surfaces $S_{\theta,1,-1}$ are graphs or not. However they are actually graphs (over close-to-convex domains) by Theorem \ref{thm:main_krust2}, since a rectifiably $1$-arcwise connected domain is exactly a convex domain.

\section{Examples}\label{sec:example}
We give here some examples to see how to apply the theorems in the previous sections.
\begin{example}\label{Ex:1}
Taking the Weierstrass data $(F, G)=(1, w^n)$ where $n$ is a positive integer, we obtain the deformation family of the Enneper-type surface
\[
X_{\theta,\lambda,c}(w)= \transpose{\left( \frac{e^{i\theta}}{\lambda}w-c\lambda e^{-i\theta}\frac{\overline{w}^{2n+1}}{2n+1}, \mathrm{Re}\left(\frac{2w^{n+1}}{n+1}\right) \right)} \in \mathbb{C}\times \mathbb{R},\quad  w\in \mathbb{D}
\]
by the equation \eqref{eq:planar_rep_X}. Let us define the planar harmonic mapping $f_{\theta, \lambda, c}$ and its boundary function $\gamma_{\theta, \lambda, c}$ as
\begin{align*}
f_{\theta, \lambda, c}(w)=w-\frac{c\lambda^2}{(2n+1)e^{i2\theta}}\overline{w}^{2n+1},\quad 
\gamma_{\theta, \lambda, c}(\varphi)= f_{\theta, \lambda, c}(e^{i\varphi}).
\end{align*}
Since $\gamma_{0,1,1}$ is a hypocycloid, see Fig.  \ref{Fig:Enneper1}, the domain of the minimal graph $X_{0,1,1}$ is not convex (but it is starlike), and we cannot apply the original Krust theorem (Theorem \ref{thm:Krust}) for the minimal graph $X_{0,1,1}$.
However, even in such a situation, $f_{0,1,0}(w)=h(w)=w$ on $\mathbb{D}$ is obviously univalent and convex. Hence, Theorem \ref{thm:main_krust2} implies that each surface $S_{\theta, \lambda, c}=X_{\theta, \lambda, c}(\mathbb{D})$ is also a graph over a close-to-convex domain under the condition $|c\lambda^2|\leq 1$, see Fig.  \ref{Fig:Enneper1} and Fig.  \ref{Fig:Enneper2}. 

Although, by Theorem \ref{thm:non_graph}, each surface $S_{\theta,\lambda, c}$ is no longer a graph when $|c\lambda^2|>1$, its restriction $S_{\theta,\lambda, c}^R=X_{\theta,\lambda, c}(\mathbb{D}_R)$ for $0<R <1$ is a graph when $1< |c\lambda^2|<1/R$ by Corollary \ref{cor:subgraph1}.

\begin{figure}[htbp]
\hspace*{-6ex}
    \begin{tabular}{ccc}
      \begin{minipage}[t]{0.55\hsize}
        \centering
        \includegraphics[keepaspectratio, scale=0.45]{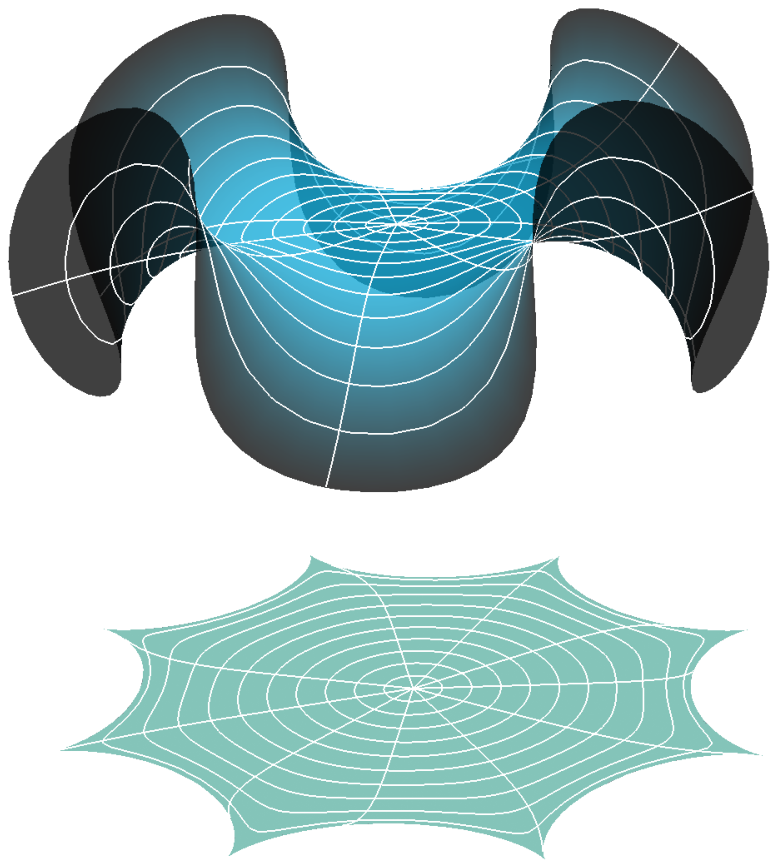}
              \end{minipage} 
           \hspace*{-20ex}
      \begin{minipage}[t]{0.50\hsize}
        \centering
        \vspace{-4.15cm}
        \includegraphics[keepaspectratio, scale=0.40]{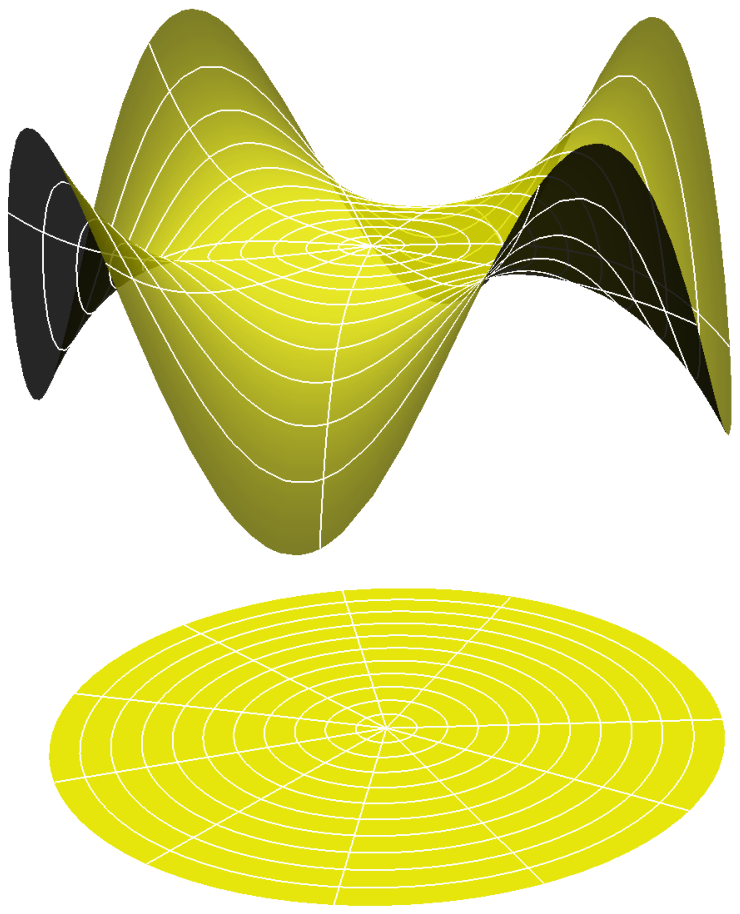}
            \end{minipage} 
             \hspace*{-20ex}
        \begin{minipage}[t]{0.55\hsize}
        \centering
        \includegraphics[keepaspectratio, scale=0.45]{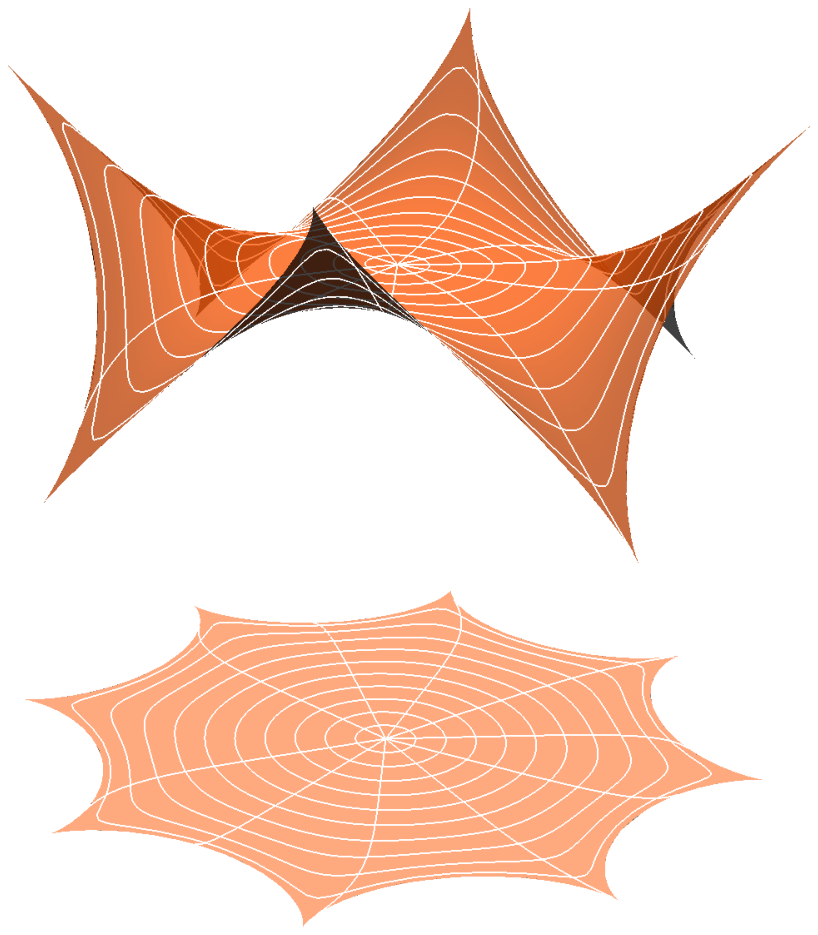}
           \end{minipage} 
    \end{tabular}
       \caption{The Enneper-type minimal graph $X_{0, 1, 1}$ with $n=3$ (left) and its deformations $X_{0, 1, 0}$ (center) and $X_{0, 1, -1}$ (right).}
       \label{Fig:Enneper1}
\end{figure}


\begin{figure}[htbp]
 \vspace{4.0cm}
\hspace*{-8ex}
    \begin{tabular}{cc}
      \begin{minipage}[t]{0.55\hsize}
        \centering
                \vspace{-3.2cm}
        \includegraphics[keepaspectratio, scale=0.45]{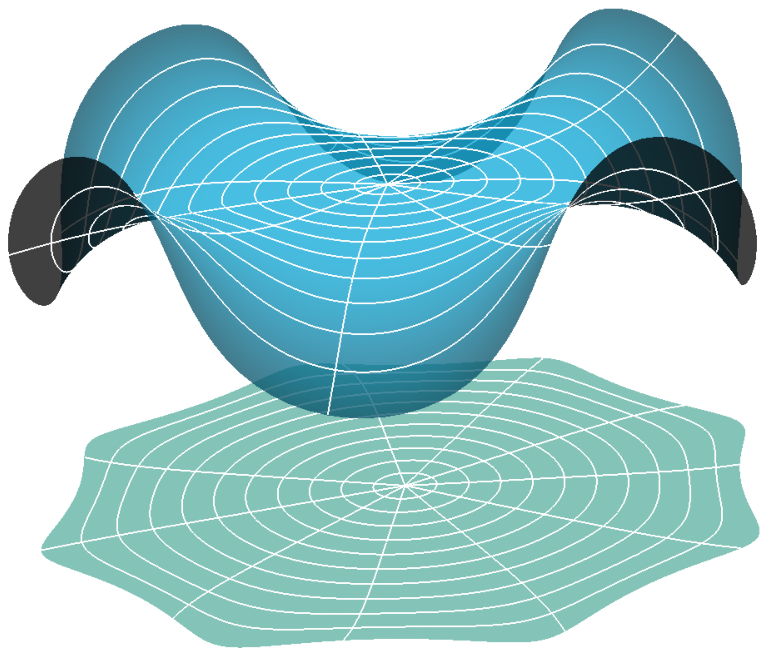}
              \end{minipage} 
                 \hspace*{-20ex}
        \begin{minipage}[t]{0.55\hsize}
        \centering
            \vspace{-3.7cm}
        \includegraphics[keepaspectratio, scale=0.45]{Fig7.eps}
              \end{minipage} 
           \hspace*{-20ex}
      \begin{minipage}[t]{0.50\hsize}
        \centering
        \vspace{-4.5cm}
        \includegraphics[keepaspectratio, scale=0.50]{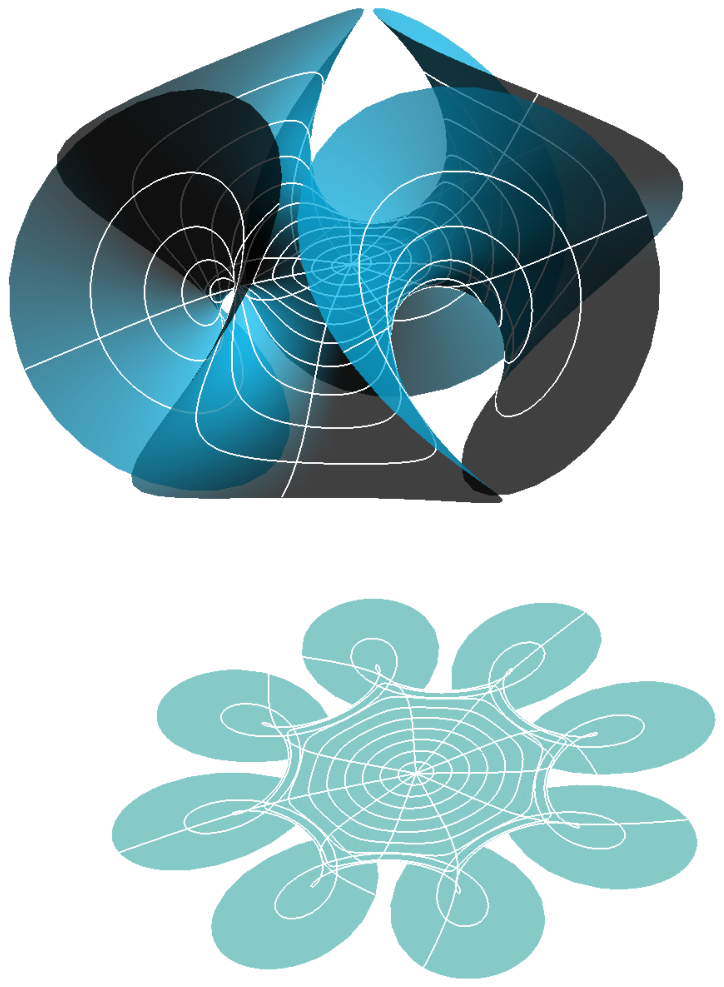}
            \end{minipage} 
    \end{tabular}
       \caption{The L\'opez-Ros deformation  $X_{0,\lambda, 1}$ of $X_{0,1,1}$ (center) defined on $\mathbb{D}$. The left one is $X_{0,0.6,1}$ and the right one is $X_{0,2,1}$, which is not a graph because it does not satisfy $\lambda^2\leq 1$ on $\mathbb{D}$ .}
       \label{Fig:Enneper2}
  \end{figure}

\end{example}

\begin{example}\label{Ex:2}
Taking the Weierstrass data $(F, G)=(1,e^{nw})$ where $n$ is a positive integer, we obtain the deformation family 
\[
X_{\theta, \lambda, c}(w)=\transpose{\left( \frac{e^{i\theta}}{\lambda}w-\frac{c\lambda}{2ne^{i\theta}}e^{2n\overline{w}}, \frac{2}{n}\mathrm{Re}(e^{nw}) \right)},\quad  \mathrm{Re}{(w)}<0
\]
by the equation \eqref{eq:planar_rep_X}. Although the domain of the minimal graph $X_{0,1,1}$ is not even starlike, the domain of the graph $X_{0,1,0}$ is convex, see Fig.  \ref{Fig:Planar1}. Hence, Theorem \ref{thm:main_krust2} implies that each surface $S_{\theta, \lambda, c}$ is also a graph over a close-to-convex domain under the condition $|c\lambda^2|\leq 1$, see Fig.  \ref{Fig:Planar1} and Fig.  \ref{Fig:Planar2}. 

\begin{figure}[htbp]
\hspace*{-6ex}
    \begin{tabular}{cc}
      \begin{minipage}[t]{0.55\hsize}
        \centering
        \includegraphics[keepaspectratio, scale=0.40]{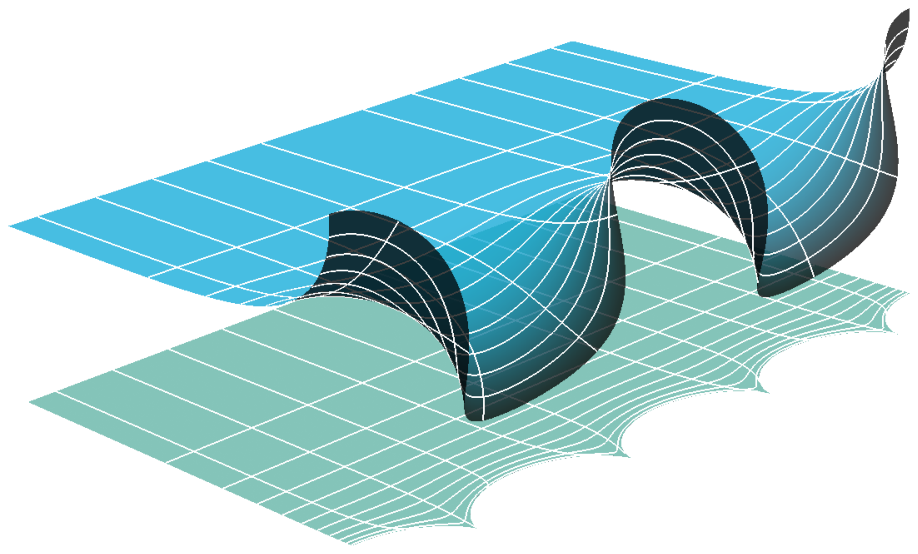}
              \end{minipage} 
           \hspace*{-22ex}
      \begin{minipage}[t]{0.55\hsize}
        \centering
        \vspace{-3.1cm}
        \includegraphics[keepaspectratio, scale=0.31]{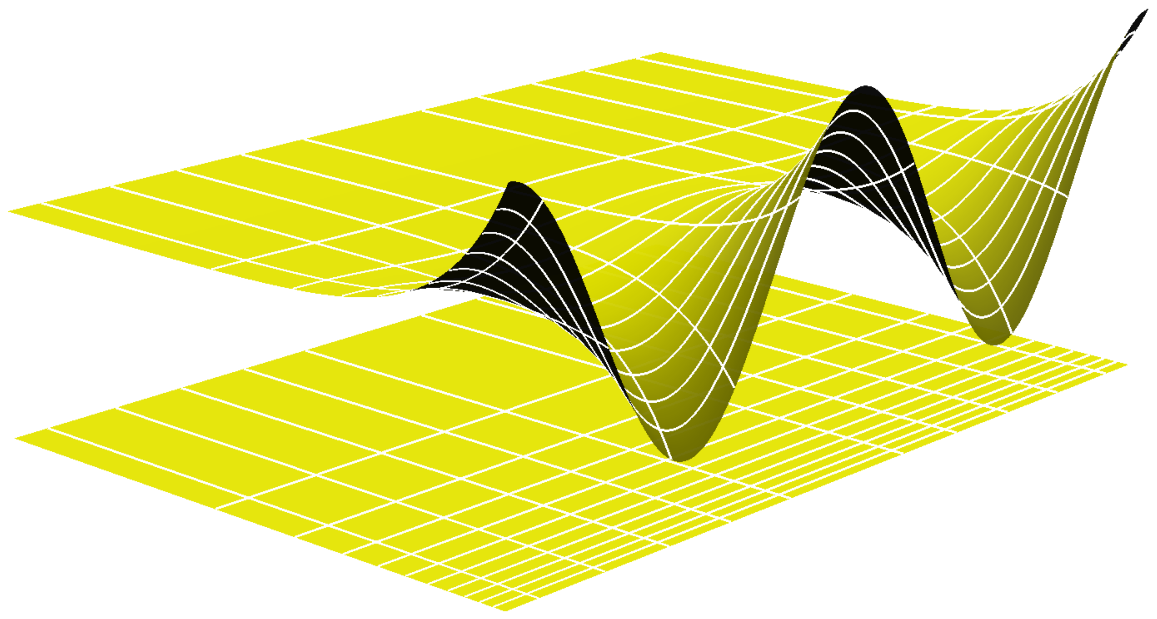}
            \end{minipage} 
             \hspace*{-22ex}
        \begin{minipage}[t]{0.55\hsize}
        \centering
                \vspace{-3.1cm}
        \includegraphics[keepaspectratio, scale=0.33]{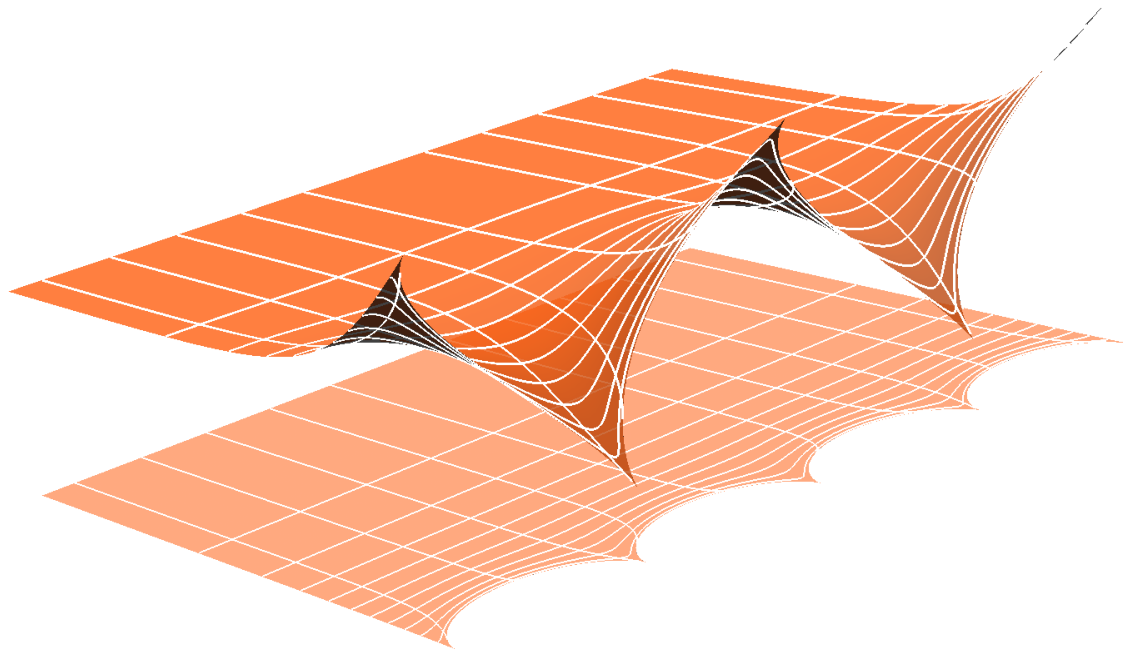}
           \end{minipage} 
    \end{tabular}
       \caption{The minimal graph $X_{0, 1, 1}$ in Example \ref{Ex:2} with $n=2$ (left) and its deformations $X_{0, 1, 0}$ (center) and $X_{0, 1, -1}$ (right).}
       \label{Fig:Planar1}
  \end{figure}

\begin{figure}[htbp]
\vspace*{5ex}
\hspace*{-8ex}
    \begin{tabular}{cc}
      \begin{minipage}[t]{0.55\hsize}
        \centering
        \includegraphics[keepaspectratio, scale=0.35]{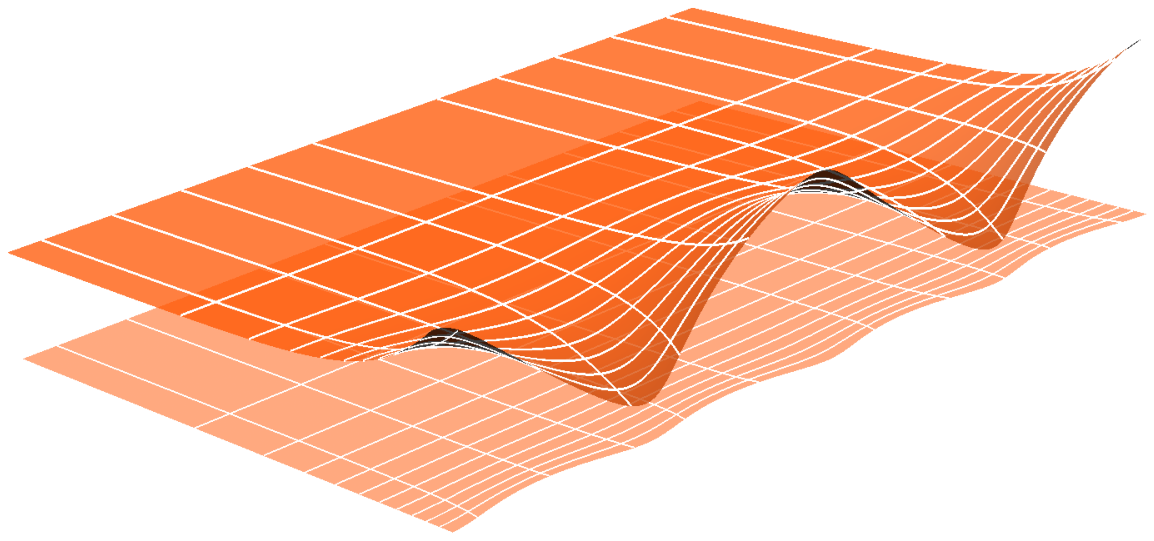}
              \end{minipage} 
                 \hspace*{-20ex}
        \begin{minipage}[t]{0.55\hsize}
        \centering
        \includegraphics[keepaspectratio, scale=0.35]{Fig14.eps}
              \end{minipage} 
           \hspace*{-20ex}
      \begin{minipage}[t]{0.55\hsize}
        \centering
        \vspace{-3.5cm}
        \includegraphics[keepaspectratio, scale=0.35]{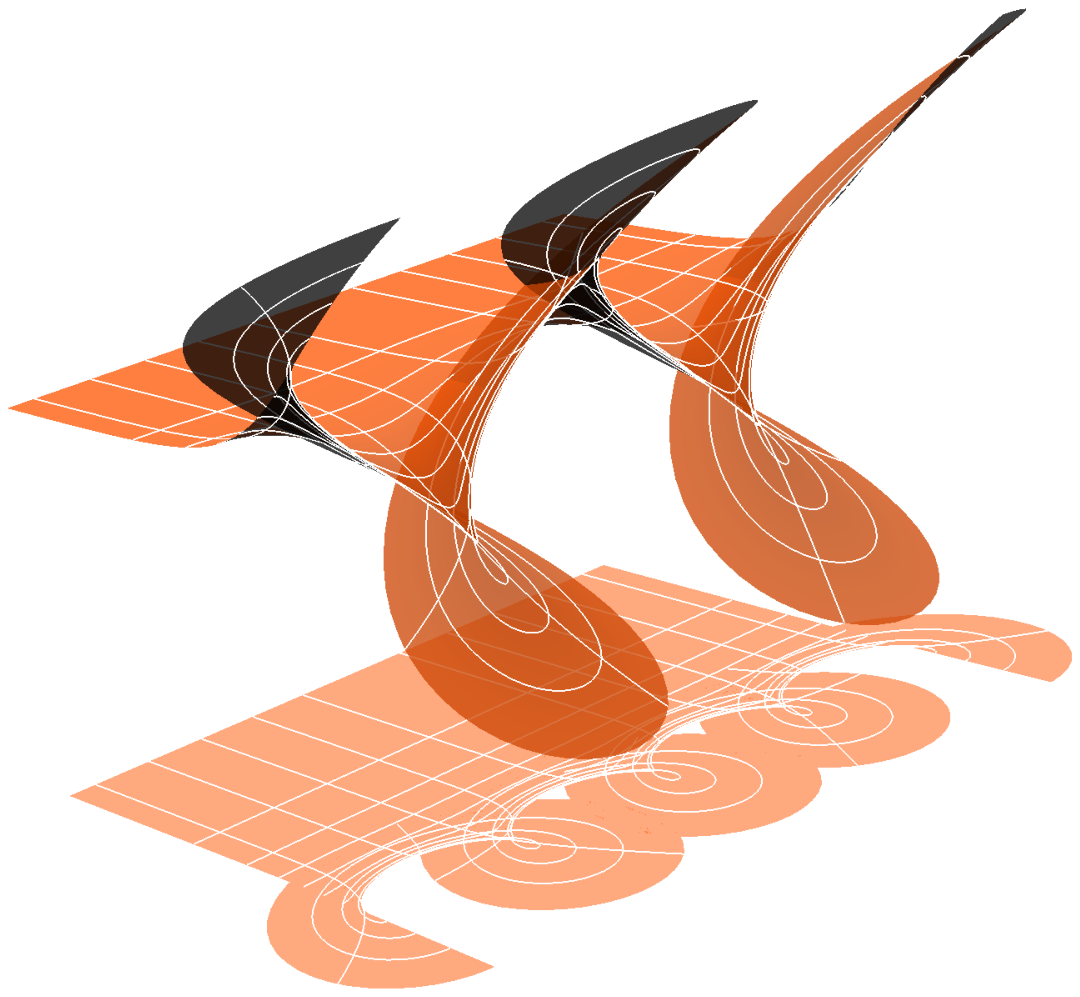}
            \end{minipage} 
    \end{tabular}
       \caption{The L\'opez-Ros deformation  $X_{0,\lambda, -1}$ of $X_{0,1,-1}$ defined on $\mathbb{D}$ (center). The left one is $X_{0,0.5,-1}$ and the right one is $X_{0,2.3,-1}$, which is not a graph because it does not satisfy $\lambda^2\leq 1$ on $\mathbb{D}$.}
       \label{Fig:Planar2}
  \end{figure}
\end{example}

\begin{example}\label{Ex:3}
Taking the Weierstrass data $(F, G)=(\frac{4}{1-w^4}, w)$ where $n$ is a positive integer, we obtain the deformation family of the Scherk surface
\[
X_{\theta, \lambda, c}(w)=\transpose{\left( \frac{e^{i\theta}}{\lambda} f_{\theta,\lambda, c}(w), 2\log{\left|\frac{1+w^2}{1-w^2}\right|} \right)},\quad  w\in \mathbb{D},
\]
where $f_{\theta,\lambda, c}$ is 
\begin{align*}
f_{\theta,\lambda, c}(w)=&\log{\left(\frac{1+w}{1-w}\right)} +i \log{\left(\frac{1-iw}{1+iw}\right)} \\
&+\frac{c\lambda^2}{e^{i2\theta}}\left\{  \log{\left(\frac{1-\overline{w}}{1+\overline{w}}\right)} -i \log{\left(\frac{1+i\overline{w}}{1-i\overline{w}}\right)} 
 \right\} . 
\end{align*}
Since the graph $X_{0,1,0}$ is defined over a convex domain, see Fig.  \ref{Fig:Scherk1}, Theorem \ref{thm:main_krust1} implies that each surface $S_{\theta, \lambda, c}=X_{\theta, \lambda, c}(\mathbb{D})$ is also a graph over a close-to-convex domain under the condition $|c\lambda^2|\leq 1$.

\begin{figure}[htbp]
\hspace*{-5ex}
    \begin{tabular}{cc}
      \begin{minipage}[t]{0.55\hsize}
        \centering
        \includegraphics[keepaspectratio, scale=0.35]{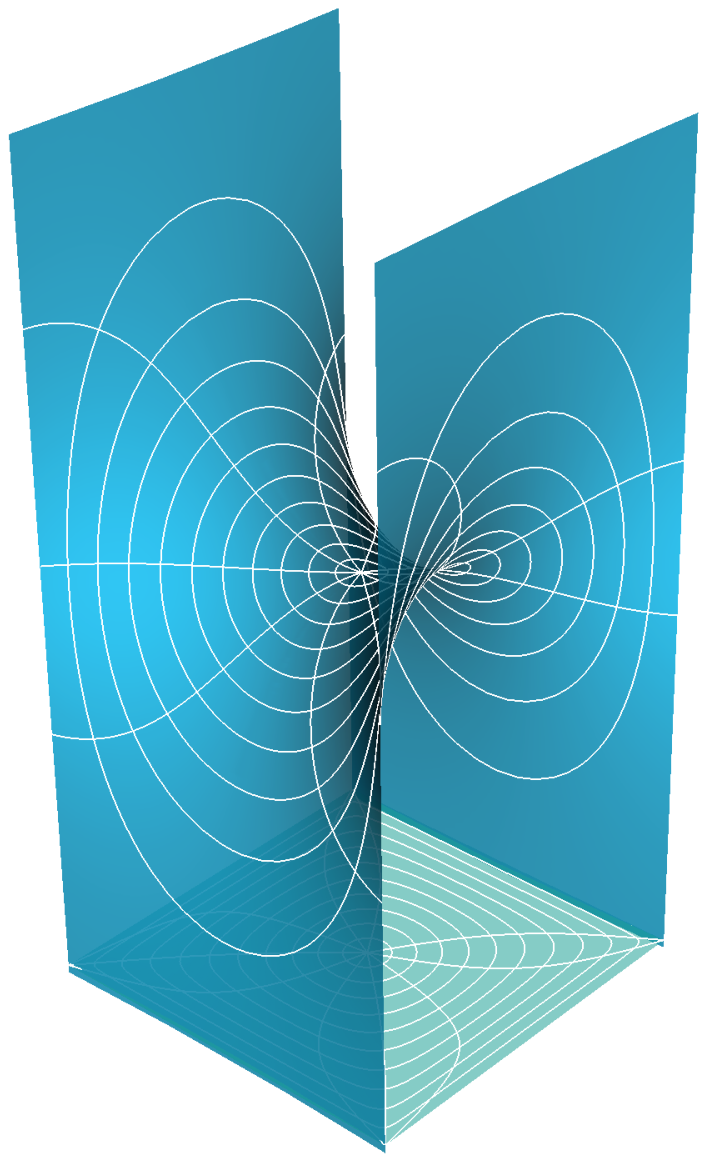}
              \end{minipage} 
           \hspace*{-24ex}
      \begin{minipage}[t]{0.55\hsize}
        \centering
        \vspace{-4.8cm}
        \includegraphics[keepaspectratio, scale=0.45]{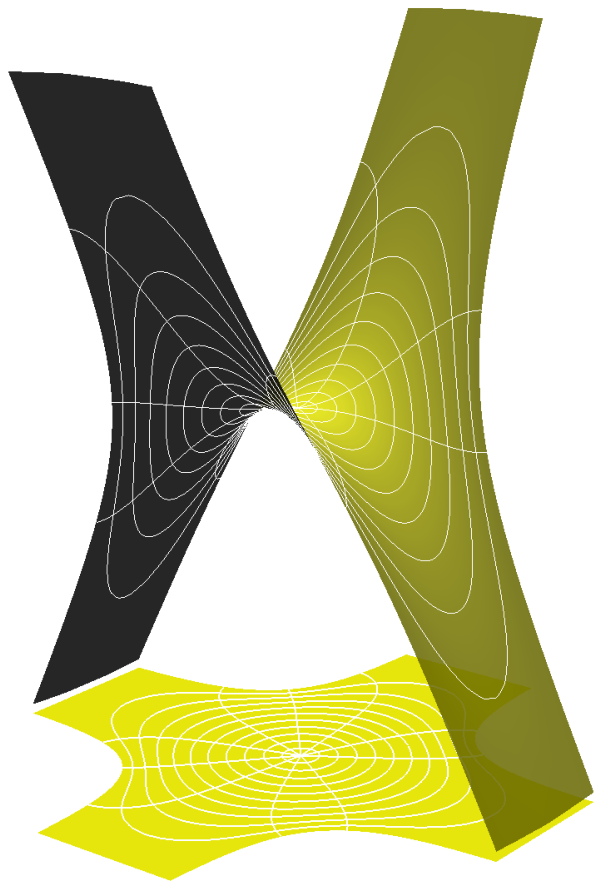}
            \end{minipage} 
             \hspace*{-15ex}
        \begin{minipage}[t]{0.55\hsize}
        \centering
             \vspace{-4.2cm}
        \includegraphics[keepaspectratio, scale=0.50]{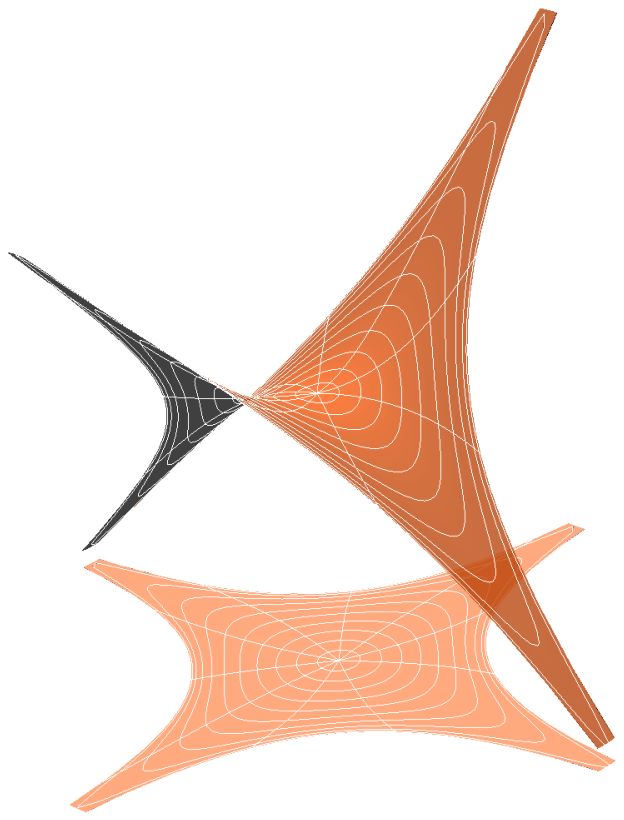}
           \end{minipage} 
    \end{tabular}
       \caption{The minimal graph $X_{0, 1, 1}$ in Example \ref{Ex:3} (left) and its deformations $X_{0, 1, 0}$ (center) and $X_{0, 1, -1}$ (right).}
       \label{Fig:Scherk1}
  \end{figure}
  
\end{example}


\begin{acknowledgement}
The authors would like to express their gratitude to Professor Ken-ichi Sakan for his helpful comments and giving us information on his article \cite{PS}, and the referee for his/her careful reading of the submitted version of the manuscript and fruitful comments and suggestions.
\end{acknowledgement}





\end{document}